\documentclass[12pt]{amsart}
\usepackage[utf8]{inputenc}
\usepackage[english]{babel}
\usepackage{tikz-cd}
\usepackage{verbatim}
\usepackage{nccmath}
\newcommand\tsum{\textstyle\sum\nolimits}

\usepackage{tabstackengine}

\title[Manin's conjecture and generically finite covers]{Manin's Conjecture and the Fujita invariant of finite covers}
\author{Akash Kumar Sengupta}
\date{}

\usepackage[
hmarginratio={1:1},     
vmarginratio={1:1},     
textwidth=410pt,        
heightrounded,          
]{geometry}
\usepackage{lipsum}
\usepackage{amsmath,amsthm,amssymb,mathrsfs}

\newtheoremstyle{common}
{6pt plus 5pt minus 2pt}
{6pt plus 5pt minus 2pt}
{\normalfont}
{0em}
{\bfseries}
{}
{.5em}
{}
\theoremstyle{common}

\newtheorem{thm}{Theorem}[section]
\newtheorem{lem}[thm]{Lemma}
\newtheorem{prop}[thm]{Proposition}

\newtheorem{defn}[thm]{Definition}
\newtheorem{conj}[thm]{Conjecture}
\newtheorem{cor}[thm]{Corollary}
\begin{document}
	
	\maketitle
	\begin{abstract}
		We prove a conjecture of Lehmann-Tanimoto about the behaviour of the Fujita invariant (or $a$-constant appearing in Manin's conjecture) under pull-back to generically finite covers. As a consequence we obtain results about geometric consistency of Manin's conjecture.
	\end{abstract}
	
	\section{Introduction}
	Let $X$ be a smooth projective variety over a field of $k$ of characteristic $0$ and $L$ a big $\mathbb{Q}$-Cartier $\mathbb{Q}$-divisor on $X$. Let $\Lambda_{\mathrm{eff}}(X) \subset \mathrm{NS}(X)_{\mathbb{R}}$ be the cone of pseudo-effective divisors. The Fujita invariant or the $a$-constant is defined as  
	\[a(X,L)= \mathrm{min}\{t\in \mathbb{R}| [K_X]+t[L] \in \Lambda_{\mathrm{eff}}(X)\}\]
	The invariant $\kappa\epsilon(X,L)=-a(X,L)$ was introduced and studied by Fujita under the name Kodaira energy in [Fuj87], [Fuj92], [Fuj96] and [Fuj97]. By [BDPP13], we know that $a(X,L)>0$ if and only if $X$ is uniruled.\\
	
	The $a$-constant was introduced in the context of Manin's conjecture in [FMT89] and [BM90]. Manin's conjecture predicts that the asymptotic behaviour of the number of rational points on Fano varieties over number fields is governed by certain geometric invariants (the $a$ and $b$-constants). In [LTT16], motivated by Manin's conjecture, the authors studied the behaviour of the $a$-constant under restriction to subvarieties. We know (by [LTT16], [HJ16]) that if $X$ is uniruled and $L$ is a big and nef line bundle then there exists a proper closed subscheme $V\subset X$ such that any subvariety $Z \subset X$ satisfying $a(Z,L|_Z)>a(X,L)$ is contained in $V$. In [LT17], a similar finiteness statement was conjectured about the behaviour of the $a$-constant under pull-back to generically finite covers and in this paper we confirm the conjecture. In particular we prove the following:
	
	\begin{thm} (see [LT17, Conj. 1.7])
		Let $X$ be a smooth projective uniruled variety and $L$ a big and nef $\mathbb{Q}$-divisor. Then, upto birational equivalence, there are only finitely many generically finite covers $f:Y \longrightarrow X$ such that $a(Y,f^*L)=a(X,L)$ and $\kappa(K_Y+a(Y,f^*L)f^*L)=0$.
	\end{thm}
	The conjecture was proved in the case of $\mathrm{dim}(X)=2$ in  [LT17]. The authors also showed that the conjecture holds for Fano threefolds $X$ and $L=-K_X$ if $\mathrm{index}(X)\geq 2$ or if $\rho(X)=1$, $\mathrm{index}(X)=1$ and $X$ is general in its moduli. Their idea was to reduce the statement to the finiteness of the \'{e}tale fundamental groups of log Fano varieties. In this paper we take a different approach to prove the conjecture in general. It follows from the boundedness results in [Bir16] that the degree of morphisms $f:Y\longrightarrow X$ satisfying the hypothesis of Theorem 1.1 is bounded. Therefore it is enough to show that the branch divisors of all such morphisms are contained in a fixed proper closed subscheme $V\subsetneq X$. We show that if $B\subset X$ is component of the branch divisor and $B \not\subset \mathbf{B}_{+}(L)$, then $a(B,L|_B)>a(X,L)$. Here $\mathbf{B}_{+}(L)$ is a closed subset of $X$ such that $L|_B$ is big for any subvariety $B\not\subset \mathbf{B}_{+}(L)$ ([Laz04]). Therefore, by [HJ16], there is a fixed (depending only on $X$ and $L$) closed subscheme $V\subsetneq X$ such that $B\subset V \cup \mathbf{B}_{+}(L)$.\\
	
	We note that the converse of the statement above is not true. In particular, if $B$ is a subvariety of $X$ with $a(B,L|_B)>a(X,L)$ then there might not exist a generically finite cover satisfying the hypothesis of Theorem 1.1 such that $B$ is contained in the branch divisor. For example, let $X=\mathrm{Bl}_{p}\mathbb{P}^2$ be the blow-up of $\mathbb{P}^2$ at a point $p$ and $B=E$ the exceptional divisor. Let $L=-K_X$. Then we have $a(X,L)=1$ and $a(B,L)=2$. By [LT17, Theorem 6.2], $X$ does not admit any generically finite cover $f:Y\longrightarrow X$ satisfying $a(Y,f^*L)=a(X,L)$ and $\kappa(K_Y+a(Y,f^*L)f^*L)=0$.\\

	As a consequence of Theorem 1.1, we can obtain a statement about geometric compatibility of Manin's conjecture. Let us recall the $b$-constant, which is the other geometric invariant involved in Manin's conjecture. The $b$-constant is defined as (cf. [FMT89], [BM90])
	\[ b(X,L) = \mathrm{codim \hspace*{0.1cm} of \hspace*{0.1cm}minimal\hspace*{0.1cm} supported\hspace*{0.1cm} face\hspace*{0.1cm} of }\hspace*{0.1cm} \Lambda_{\mathrm{eff}}(X)\hspace*{0.1cm}\] \[\mathrm{containing \hspace*{0.1cm} the\hspace*{0.1cm} class\hspace*{0.1cm} of\hspace*{0.1cm} }K_X+a(X,L)L\]
	
	In [FMT89] and [BM90], it was conjectured that the $a$ and $b$-constants control the count of rational points on Fano varieties over a number field. The following version was suggested by Peyre in [Pey03] and later stated in [Rud14], [BL15].\\
	
	{\bf Manin's conjecture:} Let $X$ be a Fano variety defined over a number field $F$ and $\mathcal{L}=(L,||.||)$ a big and nef adelically metrized line bundle on $X$ with associated height function $H_{\mathcal{L}}$. Then there exists a thin set $Z \subset X(F)$ such that  
	\[\#\{x\in X(F)-Z| H_{\mathcal{L}}(x)\leq B\} \sim c(F,X(F)\setminus Z,\mathcal{L})B^{a(X,L)}(\mathrm{log}B)^{b(X,L)-1}\]
	as $B\rightarrow \infty$. \\
	
	Recall that a thin subset of $X(F)$ is a finite union $\cup_i\pi_i(Y_i(F))$ where $\pi_i:Y_i\longrightarrow X$ is a morphism generically finite onto its image and admits no rational section. Initially Manin's conjecture was stated for closed subsets instead of thin subsets (see [BM90], [Pey95], [BT98]). But it turned out that the closed set version is false (see [BT96],[Rud14]). The counterexamples arise from the existence of generically finite morphisms $f:Y \longrightarrow X$ such that 
	\[(a(Y,f^*L),b(Y,f^*L))>(a(X,L),b(X,L))
	\]
	in the lexicographic order.\\
	
	In [LT17], it was conjectured that such geometric incompatibilities can not obstruct the thin set version of Manin's conjecture. 
	
	\begin{conj}(see [LT17, Conj1.1]) Let $X$ be a smooth projective uniruled variety over a number field $F$. Consider all $F$-morphisms $f:Y\longrightarrow X$ where $Y$ is a smooth projective variety and $f$ is generically finite. Then, as we vary over all such morphisms $f$ such that $f^*L$ is not big or 
		\[(a(Y,f^*L),b(Y,f^*L))>(a(X,L),b(X,L))
		\]
		in the lexicographic order, the points in $\cup_f f(Y(F))$ are contained in a thin subset $Z\subset X$.
	\end{conj}
	In [LT17], it was shown that the thinness statement of Conjecture 1.2 holds, when $f$ is varied over the inclusion morphism of subvarieties $Y\hookrightarrow X$ and $\rho(X)=\rho(\overline{X})$ and $\kappa(K_X+a(X,L)L)=0$. However the finiteness statement of Theorem 1.1 does not hold over number fields due to the presence of twists. But it was proved in [LT17] that, for a fixed generically finite $F$-cover $f:Y \longrightarrow X$ satisfying $\kappa(K_Y+a(X,L)L)=0$, if we vary $f^{\sigma}:Y^{\sigma}\longrightarrow X$ over all the twists of $f$, then the rational points, contributed by the twists satisfying the hypothesis of Conjecture 1.1, are contained in a thin set. Therefore as a consequence of Theorem 1.1, we obtain the following partial result towards Conjecture 1.2:
	\begin{cor}
		Let $X$ be a smooth uniruled variety over a number field $F$ such that $\rho(X)=\rho(\overline{X})$. Let $L$ be a big and nef $\mathbb{Q}$-divisor on $X$ with $\kappa(K_X+a(X,L)L)=0$. If we vary over all generically finite $F$-covers $f:Y\longrightarrow X$ with $\kappa(K_Y+a(Y,f^*L)f^*L)=0$ such that \[(a(Y,f^*L),b(Y,f^*L))>(a(X,L),b(X,L))
		\]
		then the set of rational points $\cup_f f(Y(F))$ are contained in a thin subset $Z\subset X$.\\
	\end{cor}
	
	The outline of the paper is as follows. In Section 2.1, we prove the key statements about the $a$-constant. In Section 2.2, we prove the boundedness of the degree of the morphisms in Theorem 1.1. In Section 2.3, we recall the facts about twists of morphisms over a number field. Finally, in section 3, we prove Theorem 1.1 and Corollary 1.3. \\
	
	{\bf Acknowledgements.} I am very grateful to my advisor Professor J\'{a}nos Koll\'{a}r for his constant support, encouragement and fruitful discussions. I am thankful to Professor Brian Lehmann and Professor Sho Tanimoto for useful comments on a preliminary version of the paper. Partial financial support was provided by the NSF under grant number DMS-1502236.
	
	\section{Preliminaries}
	In this paper we work in characteristic $0$.

	\subsection{Geometric invariants.}
	Let $X$ be a smooth projective variety over a field $k$. The N\'{e}ron-Severi group $\mathrm{NS}(X)$ is defined as the quotient of the group of Weil divisors, $\mathrm{Wdiv}(X)$, modulo algebraic equivalence. The pseudo-effective cone $\Lambda_{\mathrm{eff}}(X)$ is the closure of the cone of effective divisor classes in $\mathrm{NS}(X)_{\mathbb{R}}$. The interior of $\Lambda_{\mathrm{eff}}(X)$ is the cone of big divisors $\mathrm{Big}^1(X)_{\mathbb{R}}$. \\

	\begin{defn} Let $L$ be a big Cartier $\mathbb{Q}$-divisor on $X$. The $a$-constant is defined as 
		\[a(X,L)= \mathrm{min}\{t\in \mathbb{R}| [K_X+tL] \in \Lambda_{\mathrm{eff}}(X)\}\]

	\end{defn}
	If $L$ is not big, we formally set $a(X,L)=\infty$. For a singular projective variety we define $a(X,L):= a(\widetilde{X},\pi^*L)$ where $\pi: \widetilde{X}\longrightarrow X$ is a resolution of $X$. It is invariant under pull-back by a birational morphism of smooth varieties and hence independent of the choice of the resolution. By [BDPP13] we know that  $a(X,L)>0$ if and only if $X$ is uniruled. We note that the $a$-constant is independent of base change to another field. It was shown in [BCHM10] that, if $X$ is uniruled and has klt singularities, then for an ample $\mathbb{Q}$-Cartier $\mathbb{Q}$-divisor $L$, the Fujita invariant $a(X,L)$ is a rational number. If $L$ is big and not ample, then $a(X,L)$ can be irrational (see [HTT15, Example 6]). \\

	\begin{defn}
		Let $X$ be a smooth projective variety over $k$ and $L$ a big Cartier $\mathbb{Q}$-divisor. The $b$-constant is defined as
		\[ b(k,X,L) = \mathrm{codim \hspace*{0.1cm} of \hspace*{0.1cm}minimal\hspace*{0.1cm} supported\hspace*{0.1cm} face\hspace*{0.1cm} of }\hspace*{0.1cm} \Lambda_{\mathrm{eff}}(X)\hspace*{0.1cm}\] \[\mathrm{containing \hspace*{0.1cm} the\hspace*{0.1cm} class\hspace*{0.1cm} of\hspace*{0.1cm} }K_X+a(X,L)L.\]
	\end{defn}

	It is invariant under pullback by a birational morphism of smooth varieties ([HTT15]). For a singular variety $X$ we define $b(k,X,L):=b(k,\widetilde{X},\pi^*L)$, by pulling back to a resolution. In general the $b$-constant depends on the base field $k$. It is invariant under base change of algebraically closed fields.\\
	
	For the rest of this section we work over an algebraically closed field $k$.\\

	We have the following result about the behaviour of the  $a$-constant under restriction to subvarieties.
	
	\begin{thm} (see [HJ16, Theorem 1.1], [LTT16, Theorem 4.8] )
		Let $X$ be a smooth uniruled projective variety and $L$ a big and nef $\mathbb{Q}$-divisor. Then there is a proper closed subset $V\subsetneq X$ such that any subvariety $Y$ satisfying $a(Y,L|_Y)>a(X,L)$ is contained in $V$.
	\end{thm}
	
	The above result was proved in [HJ16] when $L$ is big and semi-ample. In [LTT16], it was proved assuming the weak BAB-conjecture. By [Bir16], now we know that the BAB-conjecture holds. Therefore the above result works for $L$ big and nef.\\
	
	The behaviour of the $a$-constant under pull-back to a generically finite cover is depicted in the following inequality.
	
	\begin{lem}
		Let $f: Y\longrightarrow X$ be a generically finite surjective morphism of varieties and $L$ a big $\mathbb{Q}$-Cartier $\mathbb{Q}$-divisor on $X$. Then
		
		\[a(Y,f^*L)\leq a(X,L).\]
	\end{lem}
	
	\begin{proof}
		Since the $a$-constant is computed on a resolution of singularities. We may assume $X$ and $Y$ are smooth. As $f:Y\longrightarrow X$ is generically finite, we may write
		\[K_Y=f^*K_X+R\]
		for some effective divisor $R$. Let $a(X,L)=a$. Then we have
		\[K_Y+af^*L=f^*(K_X+aL)+R.\]
		Since $R\geq 0$ and $K_X+aL$ is pseudo-effective, we see that $K_Y+af^*L$ is also pseudo-effective. Hence $a(Y,f^*L)\leq a(X,L)=a$.
		
	\end{proof}
	
	Since the $a$-constant is not necessarily rational, we need to work with $\mathbb{R}$-divisors. We recall the definition of Iitaka dimension for $\mathbb{R}$-divisors. 
	
	\begin{defn}
		Let $X$ be a smooth projective variety and $D$ be an $\mathbb{R}$-divisor on $X$. We define 
		\[\kappa(X,D)= \mathrm{lim\hspace*{0.1cm}sup}_{m\to \infty}\frac{\mathrm{log}h^0(X,\mathcal{O}_X(\lfloor mD \rfloor))}{\mathrm{log}m}\]
		
		If $X$ is a normal projective variety and $D$ an $\mathbb{R}$-Cartier divisor, then we define 
		\[\kappa(X,D)=\kappa(\widetilde{X},\pi^*D)\]
		for a resolution of singularities $\pi:\widetilde{X}\longrightarrow X$. It is easy to see that the definition is independent of the choice of the resolution. Note that it is not necessarily true that if $D\sim_{\mathbb{R}}D'$ then $\kappa(X,D)=\kappa(X,D')$.
	\end{defn}

	We form the following definition for convenience.
	
	\begin{defn} (cf. [LT17], Section 4.1) Let $X$ be a smooth uniruled variety and $L$ a big and nef $\mathbb{Q}$-Cartier $\mathbb{Q}$-divisor. We say that a morphism $f:Y\longrightarrow X$ is an adjoint-rigid  cover preserving the $a$-constant if,
		\begin{itemize}
			\item[(1)]$f:Y \longrightarrow X$ is a generically finite surjective morphism from a normal variety $Y$,
			\item[(2)]$a(Y,f^*L)=a(X,L)$,
			\item[(3)]$\kappa(K_Y+a(Y,f^*L)f^*L)=0$.\\
		\end{itemize}
		
	\end{defn}
	Note that the conditions $(1)$-$(3)$ are preserved under taking a resolution of singularities.\\

	If $X$ is a smooth surface and $E$ is a curve contracted by the $K_X$-MMP, then $K_E$ is not pseudo-effective.
	The following proposition is a generalization of this fact and it is a key step for proving Theorem 1.1. For a smooth projective uniruled variety $X$ and a big and nef divisor $L$, this proposition enables us to compare $a(X,L)$ with the $a$-constants of $L$ under restriction to the exceptional divisors contracted by a $K_X+a(X,L)L$-MMP. \\

	\begin{prop}
		Let $X$ be a normal variety with canonical singularities and $\Delta$ an effective $\mathbb{R}$-Cartier $\mathbb{R}$-divisor which is nef. Suppose $\psi:X\dashrightarrow X'$ is a minimal model for $(X,\Delta)$ obtained by a running a $K_X+\Delta$-MMP. Let $E$ be an exceptional divisor contracted by $\psi$. Let $\pi: \widetilde{E} \longrightarrow E$ be a resolution of singularities. Then $K_{\widetilde{E}}+\pi^*(\Delta|_E)$ is not pseudo-effective.\\
	\end{prop}
	
	\begin{proof} Note that it enough to prove the statement for one resolution of singularities of $E$. In particular, let $\pi: \widetilde{X}\longrightarrow X$ be a log resolution of $(X,E+\Delta)$ such that we have a morphism $\phi= \psi \circ \pi: \widetilde{X}\longrightarrow X'$. Let $\widetilde{E}=\pi^{-1}_*E$ be the strict transform. We reduce to the case when $X$ and $E$ are smooth and $\Delta$ is ample as follows. Let $H$ be a general ample divisor on $X$. Then $\psi: X\dashrightarrow X'$ is also a $K_X+\Delta+\epsilon H$-MMP for $\epsilon>0$ sufficiently small. Hence, by replacing $\Delta$ with $\Delta+\epsilon H$, we may assume that $\Delta$ is ample. Note that, since $X$ has canonical singularities, $\phi:\widetilde{X}\longrightarrow X'$ is a $K_{\widetilde{X}}+\pi^*\Delta$-minimal model. As $\pi^*\Delta$ is big and nef, we may choose $\widetilde{\Delta}\sim_{\mathbb{R}}\pi^*\Delta$ such that $(\widetilde{X},\widetilde{\Delta})$ is klt ([Xu15, Proposition 2.3]). Now $\phi:\widetilde{X}\longrightarrow X'$ is a minimal model for $(\widetilde{X},\widetilde{\Delta})$ and since minimal models of klt pairs are isomorphic in codimension one ([BCHM10, Corollary 1.1.3]), we know that $\widetilde{E}$ will be contracted by any $K_{\widetilde{X}}+\widetilde{\Delta}$-MMP. Therefore we may assume that $X$ and $E$ are smooth and $\Delta$ is ample.  We need to show that $K_E+\Delta|_E$ is not pseudo-effective.\\

		Since $\Delta$ is ample, we may choose $\Delta_0\sim_{\mathbb{R}}\Delta$ such that $(X,E+\Delta_0)$ is simple normal crossing and divisorially log terminal, $(X,\Delta_0)$ is kawamata log terminal and $(E,\Delta_0|_E)$ is canonical, by using the Bertini theorem ([Xu15, Lemma 2.2]). \\
		
		As $(X,E+\Delta_0)$ is dlt and $\Delta_0$ is ample, by [BCHM10], we may run a $K_X+E+\Delta_0$-MMP. Since we know that the ACC holds for log canonical thresholds ([HMX14]) and special termination holds for dlt flips ([BCHM, Lemma 5.1]), the
		$K_X+E+\Delta_0$-MMP terminates with a minimal model $\theta: X\dashrightarrow X_m$ by [Bir07, Theorem 1.2]. Since $E$ is contained in the negative part of the Zariski decomposition of $K_X+E+\Delta_0$, the MMP given by $\theta$ contracts $E$. Let $\theta_k: X_k\longrightarrow X_{k+1}$ be the divisorial contraction step of the $K_X+E+\Delta_0$-MMP that contracts the push-forward of $E$ on $X_k$. Let $\Theta_k= X\dashrightarrow X
		_k$ be the composition of the steps of the $K_X+E+\Delta_0$-MMP. We denote $\Delta_k={\Theta_k}_*\Delta_0$ and $E_k={\Theta_k}_*E$. Note that $E_k$ is normal ([KM98, Proposition 5.51]). By [AK17, Theorem 7], the restriction map $\Theta_k|_E:(E,\mathrm{Diff}_E\Delta_0) \dashrightarrow (E_k,\mathrm{Diff}_{E_k}\Delta_k)$ is a composition of steps of a $K_E+\mathrm{Diff}_E\Delta_0$-MMP. As $\mathrm{Diff}_E\Delta_0=\Delta_0|_E$, it is enough to show that $K_{E_k}+\mathrm{Diff}_{E_k}\Delta_k$ is not pseudo-effective. This follows from the fact that $E_k$ is covered by curves $C$ such that $(K_{E_k}+\mathrm{Diff}_{E_k}\Delta_k)\cdot C=(K_{X_k}+E_k+\Delta_k)\cdot C<0$.

	\end{proof}

	Note that the assumption about $\Delta$ being nef is necessary. For example, let $Y$ be a minimal surface and $X=\mathrm{Bl}_4(\mathrm{Bl}_y Y)$ be the blow-up of $\mathrm{Bl}_y Y$ at four distinct points $y_i\in E$, $1\leq i \leq 4$, where $E\subset \mathrm{Bl}_y Y$ is the exceptional curve corresponding to $y \in Y$. Let $E_i\subset X$ be the exceptional curve corresponding to $y_i$ for $1\leq i\leq 4$ and $E_0$ be the strict transform of $E$ on $X$. Let $\Delta= \frac{1}{2}E_1+\frac{1}{2}E_2+\frac{1}{2}E_3+\frac{1}{2}E_4$. Note that $\Delta$ is not nef as $\Delta \cdot E_4= -\frac{1}{2}$. Now $E_0$ is contracted by the $K_X+\Delta$-MMP but $\mathrm{deg}(K_{E_0}+\Delta|_{E_0})=0$.

	\begin{cor}
		Let $X$ be a smooth projective uniruled variety and $L$ a big and nef $\mathbb{Q}$-divisor. Let $f:Y\longrightarrow X$ be a generically finite cover with $Y$ smooth and $a(Y,f^*L)=a(X,L)$. Let $R\subset Y$ be a component of the ramification divisor of $f$ (i.e. the strict transform of a component of the ramification divisor for the Stein factorization of $f$) and $B$ be the component of the branch divisor on $X$ which is the image of $R$. If $R$ is contracted by a $K_Y+a(X,L)f^*L$-MMP, then 
		\[a(B,L|_B)>a(X,L)\]
	\end{cor}
	
	\begin{proof}
		We may assume that $L|_B$ is big. We have a generically finite surjective map $f|_R: R\longrightarrow B$. Therefore, by Lemma 2.4, we have $a(R,f^*L|_R)\leq a(B,L|_B)$. Now Proposition 2.7 implies that $a(R,a(X,L)f^*L)>1$ and hence $a(R,f^*L|_R)>a(X,L)$.
	\end{proof}

	\subsection{Boundedness statements}
	Let $X$ be a normal projective variety of dimension $n$ and $D$ an $\mathbb{R}$-divisor. The volume of $D$ is defined by
	\[\mathrm{vol}(X,D)= \lim_{m\to\infty} \frac{n!h^0(X,\mathcal{O}_X(\lfloor mD\rfloor))}{m^n}\]
	
	If $D$ is nef then $\mathrm{vol}(X,D)=D^n$. Also $D$ is big iff $\mathrm{vol}(X,D)>0$. The volume depends only on the numerical class $[D]\in N^1(X)$. \\

	Let $L$ be a pseudo-effective $\mathbb{Q}$-Cartier divisor on $X$. The stable base locus of $L$ ([Laz04]) is defined as
	\[\mathbf{B}(L):= \cap_{m\in \mathbb{N}} \mathrm{Bs}(mL)\]
	where the intersection is taken over $m$ such that $mL$ is Cartier. The augmented base locus of $L$ is defined as
	\[\mathbf{B}_{+}L:= \cap_{A} \mathbf{B}(L-A)
	\]
	where the intersection is over all ample $\mathbb{Q}$-Cartier divisors $A$. It is known that $\mathbf{B}_{+}(L)$ is a closed subset of $X$. If $L$ is big, then $L|_Z$ is big for any subvariety $Z\not\subset \mathbf{B}_{+}L$. We recall the following well-known result.\\
	
	\begin{lem}
		Let $f:Y\longrightarrow X$ and $g:X\longrightarrow W$ be a birational morphism of normal projective varieties and $D$ an $\mathbb{R}$-divisor on $X$.
		\begin{itemize}
			\item[(1)] $\mathrm{vol}(W,g_*D)\geq\mathrm{vol}(X,D)$,
			\item[(2)] If $D$ is $\mathbb{R}$-Cartier and $E_i$ are $f$-exceptional, then \[\mathrm{vol}(f^*D+\tsum_i a_iE_i)=\mathrm{vol}(X,D)\] for $a_i>0$.\\
			
		\end{itemize}
	\end{lem}
	
	\begin{defn}
		Let $\psi: X\dashrightarrow X'$ be a proper birational contraction (i.e. $\psi^{-1}$ does not contract any divisors) of normal quasi-projective varieties. Let $D$ be a $\mathbb{R}$-Cartier divisor such that $D'=\psi_*D$ is also $\mathbb{R}$-Cartier. We say that $\psi$ is $D$-negative if for some common resolution $p:W\longrightarrow X$ and $q:W\longrightarrow Y$, we may write
		\[p^*D=q^*D'+E\]
		where $E\geq0$ is $q$-exceptional and the support of $E$ contains the strict transform of the $\psi$-exceptional divisors.
	\end{defn}
	Recall that if $\psi: X\dashrightarrow X'$ is a $K_X+\Delta$-minimal model, then it is  $K_X+\Delta$-negative by definition. Further, if $(X,\Delta)$ is terminal and $p:W\longrightarrow X$, $q:W\longrightarrow X'$ is a common resolution, then $q:W\longrightarrow X'$ is $K_W+p^*\Delta$-negative.\\

	In general, if we have a $K_X+\Delta$-negative contraction $\psi:X \dashrightarrow X'$ of a terminal pair $(X,\Delta)$, then the pushforward $(X',\psi_*\Delta)$ might not be terminal since $\Delta$ might contain the $\psi$-exceptional divisors as components. If $\Delta$ is big and nef, then the following lemma shows that we can achieve the desired conclusion by passing to a resolution.
	This result will be used in Proposition 2.15 for proving the boundedness of degrees of adjoint-rigid covers preserving the $a$-constant.

	\begin{lem}
		Let $X$ be a normal variety with terminal singularities and $D$ a big and nef $\mathbb{R}$-Cartier divisor. Let $\psi: X\dashrightarrow X'$ be a $K_X+D$-negative contraction. Then we may choose a common resolution $p:W\longrightarrow X$ and $q:W\longrightarrow X'$ with $\widetilde{\Delta} \sim_{\mathbb{R}}p^*D$ such that $(W,\widetilde{\Delta})$ and $(X',\Delta'=\mu_*\widetilde{\Delta})$ are both terminal.
	\end{lem}
	
	\begin{proof}
		Since $D$ is big and nef, we can find $\Delta\sim_{\mathbb{R}}D$ such that $(X,\Delta)$ is terminal (see the proof of [LTT16, Theorem 2.3]). Let $p:W\longrightarrow X$ and $q:W\longrightarrow X'$ be a common log resolution
		\[\begin{tikzcd}
		& W \arrow[dl, "p"'] \arrow[dr, "q"] & \\
		X \arrow[rr,dashrightarrow, "\psi" ]& & X' \\
		\end{tikzcd}\] 
		Let $E_j$ be the $\psi$-exceptional divisors and $F_i$ the $p$-exceptional divisors. Note that the $q$-exceptional divisors are $F_i$ and the strict transforms $\widetilde{E}_j$ of the $\psi$-exceptional divisors.
		We have
		\[K_W+p_*^{-1}\Delta= p^*(K_X+\Delta)+\tsum_ia_iF_i\]
		where $a_i>0$. We may add $p$-exceptional divisors to obtain
		\[K_W+p^*\Delta= p^*(K_X+\Delta)+\tsum_ib_iF_i\]
		with $b_i>0$. Since $\psi$ is $K_X+\Delta$-negative, we may write
		\[p^*(K_X+\Delta)=q^*(K_{X'}+\Delta')+\tsum_jc_j\widetilde{E}_j+\tsum_id_iF_i\]
		where $\Delta'=\psi_*\Delta$ and $c_j>0$, $d_i\geq 0$. Therefore we have
		\[K_W+p^*\Delta=q^*(K_{X'}+\Delta')+\tsum_j\alpha_j\widetilde{E}_j+\tsum_i\beta_iF_i  \]
		with $\alpha_j$,$\beta_i>0$.
		As $p^*\Delta$ is big and nef, we may choose $\widetilde{\Delta}\sim_{\mathbb{R}}p^*\Delta$ with the coefficients of $q$-exceptional divisors in $\widetilde{\Delta}$ sufficiently small such that $(W,\widetilde{\Delta})$ is a simple normal crossing terminal pair and  
		\[K_W+q_*^{-1}q_*\widetilde{\Delta}=q^*(K_{X'}+q_*\widetilde{\Delta})+\tsum_j\alpha'_j\widetilde{E}_j+\tsum_i\beta'_iF_i \]
		with $\alpha'_j$,$\beta'_i>0$. Therefore $(X',q_*\widetilde{\Delta})$ is also terminal.

	\end{proof}

	We recall the definitions and results related to the BAB-conjecture.

	\begin{defn}Let $X$ be a normal projective variety and $\Delta$ an effective boundary $\mathbb{R}$-divisor such that $K_X+\Delta$ is $\mathbb{Q}$-Cartier. We say the $(X,\Delta)$ is $\epsilon$-log canonical (resp. $\epsilon$-klt) if for a resolution of singularities $\pi: \widetilde{X}\longrightarrow X$ with exceptional divisors $E_i$, we have $a(E_i,X,\Delta)\geq -1+\epsilon$ (resp.  $a(E_i,X,\Delta)> -1+\epsilon$) where the dicrepancies $a(E_i,X,\Delta)$ are defined by the equation
		\[K_{\widetilde{X}}+\pi^{-1}_*\Delta =\pi^*(K_X+\Delta)+a(E_i,X,\Delta)E_i.\]
	\end{defn}
	
	The following is the BAB-conjecture proved by Birkar.
	
	\begin{thm}(see [Bir16, Theorem 1.1])
		Let $n$ be a natural number and $\epsilon>0$ a real number. Then the set of projective varieties $X$ such that 
		
		\begin{itemize}
			\item[(1)] $X$ is of dimension $n$ with a boundary divsior $\Delta$ such that $(X,\Delta)$ is $\epsilon$-log canonical
			\item[(2)] $-(K_X+\Delta)$ is big and nef,
		\end{itemize}
		form a bounded family.
	\end{thm}
	
	A consequence of the above theorem is the boundedness of anticanonical volumes.
	
	\begin{cor}(Weak BAB-conjecture)
		Let $n$ be a natural number and $\epsilon>0$ a real number. There exists a constant $M(n,\epsilon)$ such that, for any normal projective variety $X$ satisfying
		\begin{itemize}
			\item[(1)] $X$ is of dimension $n$ such that there is a boundary divsior $\Delta$ with $(X,\Delta)$ is $\epsilon$-klt and $K_X$ is $\mathbb{Q}$-Cartier.
			\item[(2)] $-(K_X+\Delta)$ is ample,
		\end{itemize}
		we have 
		\[\mathrm{vol}(-K_X)<M(n,\epsilon).\]
	\end{cor}
	
	As a consequence of the Weak BAB-conjecture we obtain the following result. It shows that the degrees of all adjoint-rigid covers preserving the $a$-constant are bounded by a constant.

	\begin{prop}
		Let $X$ be a smooth uniruled variety and $L$ a big and nef $\mathbb{Q}$-divisor. Then there exists a constant $M>0$ such that, if $f:Y \longrightarrow X$ is an adjoint-rigid cover preserving the $a$-constant, then $\mathrm{deg}(f)<M$.
	\end{prop}
	
	\begin{proof}
		By Lemma 2.11, we may replace $Y$ by a resolution to assume that there exists $\Delta\sim af^*L$ with $(Y,\Delta)$ terminal and we have a morphism $\psi:Y\longrightarrow Y'$ to a minimal model $(Y',\Delta')$ with $\mathbb{Q}$-factorial terminal singularities. Now $\kappa(K_Y+\Delta)=0$ implies that $\kappa(K_{Y'}+\Delta')=0$. As $K_{Y'}+\Delta'$ is semi-ample ([BCHM, Corollary 3.9.2]), we have $K_{Y'}+\Delta'\equiv 0$. Since $\Delta'$ is big, we can wrtie $\Delta'\equiv A+E$ where $A$ is ample and $E$ is effective. Now for $0<t\ll 1$, $(Y',(1-t)\Delta'+tE)$ is terminal and 
		\[K_{Y'}+(1-t)\Delta'+tE\equiv -tA\]
		is anti-ample. Therefore $(Y',(1-t)\Delta'+tE)$ is terminal log Fano. In particular, it is $\epsilon$-klt for $\epsilon=\frac{1}{2}$. Therefore by the Weak BAB conjecture (Corollary 2.14), there exists $M>0$ such that \[\mathrm{vol}(\Delta')=\mathrm{vol}(-K_{Y'}) <M.\] Since $f:Y\longrightarrow X$ is generically finite, we have $\mathrm{vol}(af^*L)=a^n\mathrm{deg}(f)\mathrm{vol}(L)$. Now we have the following inequality
		\[a^n\mathrm{deg}(f)\mathrm{vol}(L)=\mathrm{vol}(\Delta)\leq \mathrm{vol}(\Delta')<M.\]
		Therefore the degree of $f$ is bounded.
		
	\end{proof}
	
	\subsection{Twists}
	Let us assume that the ground field is a number field $F$. Let $X$ be a smooth projective variety over $F$ and $L$ a big and nef $\mathbb{Q}$-divisor on $X$.\\
	
	Let $f:Y\longrightarrow X$ be a generically finite cover defined over $F$. A twist of $f:Y\longrightarrow X$ is a generically finite cover $f':Y'\longrightarrow X$ such that, after base change to the algebraic closure $\overline{F}$, we have an isomorphism $g:\overline{Y}\xrightarrow{\sim} Y'$ with $\overline{f}=\overline{f'}\circ g$.
	\[\begin{tikzcd}
	\overline{Y}\arrow[rr, "\sim", "g"'] \arrow[dr, "\overline{f}"'] & & \overline{Y'} \arrow[dl, "\overline{f'}"]\\
	& \overline{X} &\\
	\end{tikzcd}\]
	All the twists of $f:Y\longrightarrow X$ is parametrized by the Galois cohomology of $\mathrm{Aut}(Y/X)$. Precisely, there is a bijection between the set of isomorphism classes of twists of $f$ and the Galois cohomology group
	$H^1(\mathrm{Gal}(F),\mathrm{Aut}(Y/X))$. In view of Conjecture 1.2, even if we know the finiteness of adjoint-rigid $a$-covers over $\overline{F}$, the corresponding finiteness statement might not hold over $F$ itself due to the presence of twists. However the following result shows that, the rational points contributed by all twists of $f$, satisfying the hypothesis of Conjecture 1.2, are contained in a thin subset.
	
	\begin{thm}(see [LT17, Theorem 1.10]) Let $X$ be a smooth projective variety over a number field $F$ satisfying $\rho(\overline{X})=\rho(X)$ and let $L$ be a nef and big $\mathbb{Q}$-divisor on $X$. Suppose $f:Y\longrightarrow X$ is a generically finite $F$-cover from a smooth projective variety $Y$, satisfying $\kappa(K_Y+a(X,L)f^*L)=0$. If we vary $\sigma \in
		H^1(\mathrm{Gal}(F),\mathrm{Aut}(Y/X))$ such that the corresponding twist $f^{\sigma}:Y^{\sigma}
		\longrightarrow X$ satisfies
		\[(a(Y,f^*L),b(F,Y^\sigma,(f^\sigma)^*L))>(a(X,L),b(
		F,X,L))\]
		in the lexicographic order, then the set 
		\[\cup_\sigma f^\sigma(Y^\sigma(F))\subset X(F)\]
		is contained in a thin subset of $X(F)$.

	\end{thm}

	\section{Finiteness and thinness}
	In this section we prove the main results. \\
	
	{\bf Proof of Theorem 1.1.} Let $X$ be a smooth uniruled variety of dimension $n$ and $L$ a big and nef $\mathbb{Q}$-divisor on $X$. Suppose $a(X,L)=a$. We need to show that, upto birational equivalence, there exist finitely many varieties $Y$ that admit a morphism $f:Y\longrightarrow X$ which is an adjoint-rigid cover preserving the $a$-constant. The statement is obvious if $X$ is a curve. We assume that $n\geq 2$. By passing to a resolution it is enough to show that, upto birational equivalence, there exist finitely many smooth varieties $Y$ with  a morphism $f:Y\longrightarrow X$ which is an adjoint-rigid cover preserving the $a$-constant. By Proposition 2.15, we know that there exists a constant $M>0$ such that $\mathrm{deg}(f)<M$ for any adjoint-rigid cover preserving the $a$-constant $f:Y\longrightarrow X$. Now, for an open $U\subset X$, there are finitely many \'{e}tale covers (upto isomorphism) of $U$ of a given degree $d$. Hence it is enough to show that there is a proper closed subset $V\subsetneq X$, such that if  $f:Y\longrightarrow X$ is an adjoint-rigid cover preserving the $a$-constant and $Y$ is smooth, the branch locus of $f$ is contained in $V$.\\

	Suppose $f:Y\longrightarrow X$ is an adjoint-rigid cover preserving the $a$-constant and $Y$ is smooth. Let $Y\xrightarrow{\pi} \overline{Y} \xrightarrow{\overline{f}} X$ be the Stein factorization of $f$. Let $B\subset X$ be a component of the branch divisor of $\overline{f}$. Note that, by the Zariski-Nagata purity theorem, the branch locus is a divisor. Let $\tsum_jr_j\overline{R}_j\subset \overline{Y}$ be the ramification divisor, i.e. $K_{\overline{Y}}=\overline{f}^*K_X+\tsum_jr_j\overline{R}_j$. Let $\overline{R}\subset \overline{Y}$ be a component of the ramification divisor mapping to $B$ and $R\subset Y$ be the strict transform $\pi^{-1}_*(\overline{R})$.\\

	We have the following equation
	\[K_Y+af^*L\equiv f^*(K_X+aL)+\pi^{-1}_*(\tsum_ir_i\overline{R}_i)+\tsum_ia_iE_i\]
	where $a_i>0$. Note that, as $K_X+aL$ is pseudo-effective and $L$ is big and nef, by non-vanishing ([BCHM10, Theorem D]) we have 
	\[K_X+aL\sim_{\mathbb{R}}D\geq 0.\]
	Therefore we have 
	\[K_Y+af^*L\equiv \tsum_jc_jF_j\geq 0
	\]
	
	Now, we may find a $\Delta \equiv af^*L$ such that, $(Y,\Delta)$ is terminal. As $K_Y+\Delta$ is pseudo-effective, we can run a $K_Y+\Delta$-MMP 
	\[\psi: (Y,\Delta)\dashrightarrow (Y_1,\Delta_1)\dashrightarrow \cdots \dashrightarrow (Y_m,\Delta_m)=(Y',\Delta')\]
	to obtain a $K_Y+\Delta$-minimal model $(Y',\Delta')$.
	Since $\kappa(K_Y+\Delta)=0$, we have $K_{Y'}+\Delta'\equiv 0$. Hence the $K_Y+\Delta$-MMP contracts all components of the divisor $\tsum_jc_jF_j$. As $R=F_j$ for some $j$, Corollary 2.8 implies that $a(R,af^*L|_R)>1$ and hence 
	\[a(B,L|_B)>a=a(X,L).\]
	Therefore, by Theorem 2.3, there exists a proper closed subset $V\subsetneq X$ such that $B\subset V'=V\cup \mathbf{B}_{+}(L)$. Then, for any adjoint-rigid cover preserving the $a$-constant $f:Y\longrightarrow X$ with $Y$ smooth, the branch locus of $f$ is contained in $V'$. Therefore we have the desired conclusion.
	\\
	\qed
	\\
	
	{\bf Proof of Corollary 1.3.} We have a smooth uniruled variety $X$ over a number field $F$ such that $\rho(\overline{X})=\rho(X)$ and $\kappa(K_X+a(X,L)L)=0$ and $L$
	is a big and nef $\mathbb{Q}$-divisor on $X$. Let $f:Y\longrightarrow X$ be a generically finite $F$-cover such that $\kappa(K_Y+a(Y,f^*L)f^*L)=0$ and \[(a(Y,f^*L),b(F,Y,f^*L))>(a(X,L),b(F,X,L)).\]
	It is enough to show that the rational points contributed by the Stein factorization of $f:Y\longrightarrow X$ are contained in a fixed thin set.
	Note that, by Lemma 2.4, $a(Y,f^*L)=a(X,L)$. Therefore, the morphism $\overline{f}:\overline{Y}\longrightarrow \overline{X}$, obtained by base change to the algebraic closure of $F$, is an adjoint-rigid cover preserving the $a$-constant. Hence, by Theorem 1.1, the Stein factorizations of all such morphisms $\overline{f}:\overline{Y}\longrightarrow \overline{X}$  vary in a finite set $\mathcal{S}$. Hence we need to consider rational points contributed by twists of finitely many such Stein factorizations. So we may replace $Y$ by its Stein factorization. Further, by passing to a resolution of singularities, we may assume that $\mathrm{Bir}(\overline{Y}/\overline{X})=\mathrm{Aut}(\overline{Y}/\overline{X})$ (see the proof of [LT17, Theorem 1.10]). By applying Theorem 2.16 to each $f\in \mathcal{S}$, there is a thin subset $Z_f\subset X$ such that $\cup_\sigma f^\sigma(Y^\sigma(F))\subset Z_f$, where $\sigma$ varies over all the twists of $f$. Therefore we have
	\[\cup_ff(Y(F))=\cup_{f\in \mathcal{S}}f(Y(F)\subset \cup_{f\in \mathcal{S}} Z_f=Z\]
	where $Z\subset X$ is a thin subset and the union is taken over all the morphisms $f:Y\longrightarrow X$ satisfying the hypothesis of Corollary 1.3.\\
	\qed

	\bigskip
	
	\noindent  Princeton University, Princeton NJ 08544-1000
	
	{\begin{verbatim} 
		akashs@math.princeton.edu
		\end{verbatim}}


\begin{thebibliography}{KKMSD73}
		
		
		\bibitem[AK17]{} Florin Ambro and J\'{a}nos Koll\'{a}r, \emph{Minimal models of semi-log-canonical pairs}, 2017. Preprint, 	arXiv:1709.03540
		
		\bibitem[Bir07]{} Caucher Birkar, \emph{Ascending chain condition for log canonical thresholds and termination
			of log flips}, Duke Math. J. 136 (2007), no. 1, 173–180
		
		\bibitem[Bir16]{} C. Birkar. \emph{Singularities of linear systems and boundedness of Fano varieties.} 2016. Preprint. 	arXiv:1609.05543 [math.AG]
		
		
		\bibitem[BCHM10]{} C. Birkar, P. Cascini, C. D. Hacon, and J. McKernan. \emph{Existence of minimal models for
			varieties of log general type}. J. Amer. Math. Soc., 23(2):405–468, 2010.
		
		\bibitem[BDPP13]{}
		S\'{e}bastien Boucksom, Jean-Pierre Demailly, Mihai Paun, and Thomas Peternell. \emph{The pseudoeffective
			cone of a compact K\"{a}hler manifold and varieties of negative Kodaira dimension}. J.
		Algebraic Geom., 22(2):201–248, 2013.
		
		\bibitem[BL15]{}
		T. Browning, D. Loughran,  \emph{Varieties with too many rational points} Math. Z. (2017) Online publication.
		
		\bibitem[BM90]{}V. V. Batyrev and Yu. I. Manin. \emph{Sur le nombre des points rationnels de hauteur born\'{e} des vari\'{e}t\'{e}s alg\'{e}briques.} Math. Ann., 286(1-3):27–43, 1990. 
		
		\bibitem[BT96]{}V. V. Batyrev and Y. Tschinkel. \emph{Rational points on some Fano cubic bundles.} C. R. Acad. Sci. Paris S\'{e}r. I Math., 323(1):41–46, 1996. 
		
		\bibitem[BT98]{} V. V. Batyrev and Y. Tschinkel. \emph{Tamagawa numbers of polarized algebraic
			varieties.} Ast\'{e}risque, (251):299–340, 1998. Nombre et r\'{e}partition de points de
		hauteur born\'{e}e (Paris, 1996).
		
		
		
		\bibitem[FMT89]{} J. Franke, Y. I. Manin, and Y. Tschinkel. Rational points of bounded height
		on Fano varieties. Invent. Math., 95(2):421–435, 1989.
		
		\bibitem[Fuj87]{} T. Fujita. \emph{On polarized manifolds whose adjoint bundles are not semipositive.}
		In Algebraic geometry, Sendai, 1985, volume 10 of Adv. Stud. Pure Math.,
		pages 167–178. North-Holland, Amsterdam, 1987.
		
		\bibitem[Fuj92]{} T. Fujita. \emph{On Kodaira energy and adjoint reduction of polarized manifolds.}
		Manuscripta Math., 76(1):59–84, 1992.
		
		\bibitem[Fuj96]{} T. Fujita.\emph{ On Kodaira energy of polarized log varieties.} J. Math. Soc. Japan,
		48(1):1–12, 1996.
		
		\bibitem[Fuj97]{} T. Fujita. \emph{On Kodaira energy and adjoint reduction of polarized threefolds.}
		Manuscripta Math., 94(2):211–229, 1997.
		
		\bibitem[HMX14]{} Christopher D. Hacon, James McKernan, Chenyang Xu, \emph{ACC for log canonical thresholds} Annals of Mathematics,  Volume 180, 2014.
		
		
		
		\bibitem[HJ16]{}Christopher Hacon and Chen Jiang. \emph{On Fujita invariants of subvarieties of a uniruled variety}, 2016. arXiv:1604.01867.
		
		\bibitem[HTT15]{}
		Brendan Hassett, Sho Tanimoto, and Yuri Tschinkel. \emph{Balanced line bundles and equivariant
			compactifications of homogeneous spaces}. Int. Math. Res. Not. IMRN, (15):6375–6410, 2015.
		
		\bibitem[KM98]{} J. Koll\'{a}r and S. Mori. \emph{Birational geometry of algebraic varieties}, Cambridge
		tracts in mathematics, vol. 134, Cambridge University Press, 1998.
		
		\bibitem[Laz04]{}R. Lazarsfeld, \emph{Positivity in algebraic geometry I}, Ergebnisse der Mathematik
		und ihrer Grenzgebiete. 3. Folge, vol. 49, Springer-Verlag, Berlin, 2004.
		
		
		\bibitem[LTT16]{}
		B. Lehmann, S. Tanimoto, and Y. Tschinkel. \emph{Balanced line bundles on Fano varieties}. J. Reine
		Angew. Math., 2016. Online publication.
		
		\bibitem[LT17]{}
		B. Lehmann, S. Tanimoto, \emph{On the geometry of thin exceptional sets in Manin's Conjecture},  Duke Math. J, 2017, To appear.
		
		
		
		\bibitem[Pey03]{} Emmanuel Peyre. \emph{Points de hauteur born\'{e}e, topologie ad\'{e}lique et mesures de Tamagawa}. J. Th\'{e}or. Nombres Bordeaux, 15(1):319–349, 2003.
		
		\bibitem[Rud14]{}
		C\'{e}cile Le Rudulier. \emph{Points alg\'{e}briques de hauteur born\'{e}e sur une surface}, 2014.
		http://cecile.lerudulier.fr/Articles/surfaces.pdf.
		
		\bibitem[Xu15]{} Chenyang Xu. \emph{On base point free theorem of threefolds in positive characteristic}, Journal of the Institute of Mathematics of Jussieu, 14(3), 577-588, 2015.
		
	\end{thebibliography}
\end{document}